\documentclass[12pt,a4paper]{article}
\usepackage[cm]{fullpage}
\pdfoutput=1
\date{}

\usepackage[american]{babel}

\usepackage{cite}
\usepackage{url}
\usepackage{epsfig} 
\usepackage{times} 
\usepackage{amsmath}
\usepackage{amsfonts}
\usepackage{amssymb}  

\usepackage{xcolor}
\usepackage{transparent}
\usepackage{graphicx}
\graphicspath{{media/}}
\usepackage{tikz}
\usepackage{epstopdf}
\usepackage{units}
\usepackage{pgfplots}
\usepackage{algorithm}
\usepackage{algpseudocode}

\usepackage{pstool}

\graphicspath{{media/}}
\newcommand{\executeiffilenewer}[3]{%
 \ifnum\pdfstrcmp{\pdffilemoddate{#1}}%
 {\pdffilemoddate{#2}}>0%
 {\immediate\write18{#3}}\fi%
}

\usepackage{arydshln}
\usepackage{cleveref}

\newtheorem{theorem}{Theorem}[section]

\newtheorem{proposition}[theorem]{Proposition}

\newenvironment{proof}[1][Proof]{\begin{trivlist}
\item[\hskip \labelsep {\bfseries #1}]}{\end{trivlist}}

\newcommand{\qed}{\nobreak \ifvmode \relax \else
      \ifdim\lastskip<1.5em \hskip-\lastskip
      \hskip1.5em plus0em minus0.5em \fi \nobreak
      \vrule height0.5em width0.5em depth0.25em\fi}

\newcommand{\TT}{\ensuremath{\mathsf{\tiny{T}}}}
\newcommand{\T}{^{\TT}}

\newcommand{\diff}[1][]{\mathrm{d}#1}
\newcommand{\dt}{\diff t }

\author{Michael Muehlebach and Raffaello D'Andrea
\thanks{Michael Muehlebach and Raffaello D'Andrea are with the Institute for Dynamic Systems and Control, ETH Zurich. The contact author is Michael Muehlebach, {\tt\small michaemu@ethz.ch}.
This work was supported by ETH-Grant ETH-48 15-1.}%
}
\title{\LARGE \bf
Approximation of Continuous-Time Infinite-Horizon Optimal Control Problems Arising in Model Predictive Control - Supplementary Notes}

\begin{document}

\maketitle

\begin{abstract}
These notes present preliminary results regarding two different approximations of linear infinite-horizon optimal control problems arising in model predictive control. 
Input and state trajectories are parametrized with basis functions and a finite dimensional representation of the dynamics is obtained via a Galerkin approach. It is shown that the two
approximations provide lower, respectively upper bounds on the optimal cost of the underlying infinite dimensional optimal control problem. These bounds get tighter as the number of basis functions is increased. In addition, conditions guaranteeing convergence to the cost of the underlying problem are provided.
\end{abstract}

\section{Introduction}
Model predictive control (MPC) takes input and state constraints fully into account and is therefore a promising control strategy with various applications. The standard MPC approach relies on discrete dynamics and a finite prediction horizon, which leads inevitably to issues related to closed-loop stability. In \cite{parametrizedMPC}, an approximation of the underlying infinite dimensional infinite-horizon optimal control problem has been proposed, which is based on a parametrization of input and state trajectories with basis functions. The infinite prediction horizon is maintained, and therefore closed-loop stability and recursive feasibility arise naturally from the problem formulation. Moreover, it is conjectured that the underlying infinite dimensional optimization problem is well-approximated even with a low basis function complexity.

Herein, we compare the approach from \cite{parametrizedMPC} to a different finite dimensional approximation. We analyze both with respect to convergence of the optimal costs as the number of basis functions is increased. In particular, the optimal cost of the approximation given in \cite{parametrizedMPC} decreases monotonically and approaches the cost of underlying infinite dimensional problem from above. It is shown that the corresponding optimal trajectories are guaranteed to converge and that the second approximation approaches the optimal cost of the infinite dimensional problem from below. In addition, we will establish conditions guaranteeing convergence of both approximations to the cost of the underlying infinite dimensional problem.

This report focuses on the technical proofs and complements \cite{parametrizedCDC}, where the underlying ideas are discussed in detail and a numerical example is provided.
\section{Problem Formulation}\label{Sec:ProbForm}
We present and analyze two approximations of the following optimal control problem,
\begin{align}
\begin{split}
J_\infty:=&\inf \frac{1}{2} ||x||_2^2 + \frac{1}{2} ||u||_2^2 \\
&\text{s.t.}~\dot{x}(t)=A x(t)+B u(t),~x(0)=x_0,\\
& \quad x(t)\in \mathcal{X}, \quad u(t) \in \mathcal{U}, \quad \forall t\in [0,\infty),\\
& \quad x \in L^2_n, \quad u \in L^2_m, \quad \dot{x} \in L^2_n,
\end{split}\label{eq:Jinf}
\end{align}
where $\mathcal{X}$ and $\mathcal{U}$ are closed and convex subsets of $\mathbb{R}^n$ and $\mathbb{R}^m$, respectively, containing $0$; the space of square integrable functions mapping from $[0,\infty)$ to $\mathbb{R}^q$ is denoted by $L^2_q$, where $q$ is a positive integer; and the $L_q^2$-norm is defined as
\begin{equation}
L^2_q \rightarrow \mathbb{R}, \quad x \rightarrow ||x||_2^2:=\int_{0}^{\infty} x\T x~\dt,
\end{equation}
where $\dt$ denotes the Lebesgue measure. We assume that $J_\infty$ is finite and that the corresponding minimizers, $x$ and $u$, are unique.

Input and state trajectories will be approximated as a linear combination of basis functions $\tau_i \in L^2_1$, $i=1,2,\dots$, that is
\begin{align}
\begin{split}
\tilde{x}^s(t, \eta_x):=(I_n \otimes \tau^s(t))\T \eta_x, \\ \tilde{u}^s(t, \eta_u):=(I_m \otimes \tau^s(t))\T \eta_u,
\end{split}
\end{align}
where $\otimes$ denotes the Kronecker product, $\eta_x\in \mathbb{R}^{ns}$ and $\eta_u\in \mathbb{R}^{ms}$ are the parameter vectors, and $\tau^s(t):=(\tau_1(t),\tau_2(t),\dots, \tau_s(t))\in \mathbb{R}^s$. In order to simplify notation we will omit the superscript $s$ in $\tau^s$, $\tilde{x}^s$, and $\tilde{u}^s$, and simply write $\tau$, $\tilde{x}$, and $\tilde{u}$ whenever the number of basis functions is clear from the context. Similarly, the dependence of $\tilde{x}$ and $\tilde{u}$ on $\eta_x$ and $\eta_u$ is frequently omitted. Without loss of generality we assume that the basis functions are orthonormal. Note that orthonormal basis functions can be constructed with the Gram-Schmidt procedure, \cite[p.~50]{CourantHilberMethods}.

As motivated in \cite{parametrizedMPC}, the following additional assumptions on the basis functions are made:
\begin{itemize}
\item[A1)] They are linearly independent.
\item[A2)] They fulfill $\dot{\tau}(t)=M \tau(t)$ for all $t\in [0,\infty)$, for some matrix $M \in \mathbb{R}^{s \times s}$. The eigenvalues of $M$ have strictly negative realparts.
\end{itemize}

%

\subsection{Resulting optimization problems}
In \cite{parametrizedMPC}, the following finite dimensional approximation of the original problem \eqref{eq:Jinf} is introduced,
\begin{align}
\begin{split}
J_s :=&\inf \frac{1}{2} ||\tilde{x}||_2^2 + \frac{1}{2} ||\tilde{u}||_2^2 \\
&\text{s.t.} \int_{0}^{\infty} (I_n \otimes \tau) \left(A \tilde{x} + B \tilde{u} - \dot{\tilde{x}} \right) \dt=0, \\
&\quad \tilde{x}(0)-x_0=0, \quad \eta_x \in \mathcal{X}^s, \eta_u \in \mathcal{U}^s.
\end{split}\label{eq:Prob1}
\end{align}
Note that the subscript $s$ refers to the number of basis functions used. More precisely, the optimization problem with optimal cost $J_s$ corresponds to the case where input and state trajectories are spanned by the first $s$ basis functions. The trajectories $\tilde{x}$ and $\tilde{u}$, which satisfy the equality constraint, fulfill the equations of motion exactly, that is, $A \tilde{x}(t) + B \tilde{u}(t) =\dot{\tilde{x}}(t)$ for all $t\in [0,\infty)$ and $\tilde{x}(0)=x_0$, see \cite{parametrizedMPC}. This is needed to guarantee that the minimizers of \eqref{eq:Prob1} achieve the cost $J_s$ on the nominal system. We make the following assumptions on the sets ${\mathcal{X}}^s$ and ${\mathcal{U}}^s$\footnote{The assumptions are listed for the state constraints $\mathcal{X}$ and are analogous for the input constraints $\mathcal{U}$.}
\begin{itemize}
\item[B0)] $\mathcal{X}^s$ is closed and convex
\item[B1)] $i_s(\mathcal{X}^s) \subset \mathcal{X}^{s+1}$
\item[B2)] $\eta_x \in \mathcal{X}^s$ implies $(I_n \otimes \tau(t))\T \eta_x \in \mathcal{X}$ for all $t\in [0,\infty)$,
\end{itemize}
where the inclusion $i_s$, mapping from $\mathbb{R}^{ns}$ to $\mathbb{R}^{n(s+1)}$ is defined by
\begin{align}
\tilde{x}^s(t,\eta_x)=\tilde{x}^{s+1}(t,i_s(\eta_x)),
\forall t\in [0,\infty),~\forall \eta_x \in \mathbb{R}^{ns}.
\label{eq:defi_s}
\end{align}
Assumption B0) implies that the optimization problem \eqref{eq:Prob1} is convex, and that corresponding minimizers exist, provided the existence of feasible trajectories. Assumption B1) is used to show that the cost $J_s$ is monotonically decreasing in $s$, whereas Assumption B2) implies that $J_s$ is bounded below by $J_\infty$, see Sec.~\ref{Sec:MainResults}. In the context of MPC, Assumption B2) guarantees recursive feasibility and closed-loop stability, \cite{parametrizedMPC}. In addition, the cost $J_s$ is achieved on the nominal system, as the resulting input and state trajectories respect input and state constraints and fulfill the dynamics exactly.

The following alternative approximation is introduced,
\begin{align}
\begin{split}
\tilde{J}_s :=& \inf \frac{1}{2} ||\tilde{x}||_2^2 + \frac{1}{2} ||\tilde{u}||_2^2  \\
&~\text{s.t.} \int_{0}^{\infty} (I_n \otimes \tau) \left( A \tilde{x} + B \tilde{u} - \dot{\tilde{x}} \right) \dt 
- (I_n \otimes \tau(0)) \left(  \tilde{x}(0)-x_0 \right) = 0,\\
& \qquad \eta_x \in \tilde{\mathcal{X}}^s, \eta_u \in \tilde{\mathcal{U}}^s,
\end{split}\label{eq:Prob2}
\end{align}
whose purpose is to provide a monotonically increasing sequence $\tilde{J}_s$ bounding $J_\infty$ from below. To that extent, the following assumptions on the sets $\tilde{\mathcal{X}}^s$ and $\tilde{\mathcal{U}}^s$ are made:\footnote{The assumptions are again listed for the state constraints $\mathcal{X}$ and are analogous for the input constraints $\mathcal{U}$.}
\begin{itemize}
\item[C0)] $\tilde{\mathcal{X}}^s$ is closed and convex.
\item[C1)] $\pi_s(\tilde{\mathcal{X}}^{s+1}) \subset \tilde{\mathcal{X}}^s$.
\item[C2)] For each $x \in L^2_n$ with $x(t) \in \mathcal{X}$ for all $t\in [0,\infty)$ it holds that $\pi^s(x) \in \tilde{\mathcal{X}}^s$,
\end{itemize}
where the projections $\pi_s$ and $\pi^s$ are defined as
\begin{align}
\pi^s: &L^2_n \rightarrow \mathbb{R}^{ns}, x \rightarrow \int_{0}^{\infty} (I_n \otimes \tau^s) x~\dt,\label{eq:defpi^s}\\
\pi_s: &\mathbb{R}^{n(s+1)} \rightarrow \mathbb{R}^{ns},~\eta_x \rightarrow \pi^s\left(\tilde{x}^{s+1}(\cdot,\eta_x)\right).
\label{eq:defpi_s}
\end{align}
Assumption C0) ensures that the optimization problem \eqref{eq:Prob2} is convex, and that corresponding minimizers exist, provided the existence of feasible trajectories. Assumption C1) is used to demonstrate that $\tilde{J}_s$ is monotonically increasing, whereas Assumption C2) implies that $\tilde{J}_s$ is bounded above by $J_\infty$, see Sec.~\ref{Sec:MainResults}. Examples fulfilling Assumptions B0)-B2) and C0)-C2) are provided in \cite{parametrizedCDC}.
\section{Main results}\label{Sec:MainResults}
In the following we will analyze the two approximations \eqref{eq:Prob1} and \eqref{eq:Prob2} and prove the following result:
\begin{theorem}\label{Thm:Main}
Let $N_0$ be such that $J_{N_0}$ is finite.\\
1) If Assumptions B0), B1), and B2) hold, then the sequence $J_s$ is monotonically decreasing for $s \geq N_0$, converges as $s\rightarrow \infty$, and is bounded below by $J_\infty$. The corresponding optimizers $\tilde{x}$ and $\tilde{u}$ converge strongly in $L^2_n$, respectively $L^2_m$ as $s\rightarrow \infty$.\\
2) If Assumptions C0), C1), and C2) are fulfilled, then 
$\tilde{J}_s$ is monotonically increasing for $s\geq 1$, converges as $s \rightarrow \infty$, and is bounded above by $J_\infty$.
\end{theorem}

We start by summarizing the results from \cite{parametrizedMPC}, stating that the optimal cost of \eqref{eq:Prob1} is a monotonically decreasing sequence providing an upper bound on the optimal cost of \eqref{eq:Jinf}.
\begin{proposition}\label{Prop:AMonDec}
Let $N_0$ be such that $J_{N_0}$ is finite. If Assumptions B0) and B1) are fulfilled, then the sequence $J_s$ is monotonically decreasing for $s \geq N_0$ and converges as $s \rightarrow \infty$. If Assumptions B0) and B2) are fulfilled, $J_s$ is bounded below by $J_\infty$.
\end{proposition}
\begin{proof} See \cite{parametrizedMPC}.  \hfill \qed \end{proof}
%

The fact that $J_s$ converges can be used to demonstrate convergence of the optimizer $\tilde{x}^s$ and $\tilde{u}^s$, as well as the corresponding parameters $\eta_{x_s}$ and $\eta_{u_s}$. Therefore, the parameter vectors $\eta_{x_s}$ and $\eta_{u_s}$ are interpreted as square summable sequences, i.e. as elements in $\ell^2$.\footnote{The set of square summable sequences is denoted by $\ell^2$.}
\begin{proposition}\label{Prop:AStrongConv}
Let $N_0$ be such that $J_{N_0}$ is finite and let Assumptions B0) and B1) be fulfilled. Then, the minimizers of \eqref{eq:Prob1} converge (strongly) in $\ell^2$ as $s \rightarrow \infty$, and the corresponding trajectories $\tilde{x}^s$ and $\tilde{u}^s$ converge (strongly) in $L^2_n$, respectively $L^2_m$.
\end{proposition}
\begin{proof}
We fix $s \geq N_0$ and denote the minimizers corresponding to $J_s$ by $\eta_{x_s}$, $\eta_{u_s}$, and the minimizers corresponding to $J_{s+1}$ by $\eta_{x_{s+1}}$, $\eta_{u_{s+1}}$, which we consider to be elements of $\ell^2$. The following observation can be made: The vectors $\eta_{x}:=\lambda i_s(\eta_{x_s}) + (1-\lambda) \eta_{x_{s+1}}$ and $\eta_{u}:=\lambda i_s(\eta_{u_s}) + (1-\lambda) \eta_{u_{s+1}}$ with $\lambda \in [0,1]$ are feasibly candidates for the optimization problem \eqref{eq:Prob1} over $s+1$ basis functions.\footnote{Notation is slightly abused, since $i_s$ is used to denote both the inclusion $\mathbb{R}^{ns} \rightarrow \mathbb{R}^{n(s+1)}$ and $\mathbb{R}^{ms} \rightarrow \mathbb{R}^{m(s+1)}$, which is defined in analogy to \eqref{eq:defi_s}.} This is because the constraints $i_s(\eta_{x}) \in \mathcal{X}^{s+1}$, $i_s(\eta_u) \in \mathcal{U}^{s+1}$ are satisfied by Assumptions B0) (convexity) and B1). Moreover, as the dynamics are fulfilled exactly it follows for $\tilde{x}:=(I_n \otimes \tau)\T \eta_x$ and $\tilde{u}:=(I_m \otimes \tau)\T \eta_u$ that
\begin{align*}
\tilde{x}(t)&=\lambda \tilde{x}^s(t) + (1-\lambda) \tilde{x}^{s+1}(t)\\
&=e^{At} x_0 + \int_{0}^{t} e^{A(t-\hat{t})} B (\lambda \tilde{u}^s(\hat{t})+ (1-\lambda) \tilde{u}^{s+1}(\hat{t})) \diff \hat{t}\\
&=e^{At} x_0 + \int_{0}^t e^{A(t-\hat{t})} B \tilde{u}(\hat{t}) \diff \hat{t},
\end{align*}
where $\tilde{x}^{s}:=(I_n \otimes \tau)\T \eta_{x_s}$, $\tilde{u}^s:=(I_n \otimes \tau)\T \eta_{u_s}$, $\tilde{x}^{s+1}:=(I_n \otimes \tau)\T \eta_{x_{s+1}}$, and $\tilde{u}^{s+1}:=(I_n \otimes \tau)\T \eta_{u_{s+1}}$, which concludes that the equality constraint is likewise fulfilled. Since $\eta_{x_{s+1}}$ and $\eta_{u_{s+1}}$ are the minimizers corresponding to $J_{s+1}$, we have that
\begin{equation}
J_{s+1} \leq \frac{1}{2} ||\tilde{x}||_2^2 + \frac{1}{2} ||\tilde{u}||_2^2,
\end{equation}
for all $\lambda \in [0,1]$.
The objective function is quadratic, and hence strongly convex with respect to the $L^2_n$ and $L^2_m$-norm, which leads to
\begin{align}
J_{s+1} \leq \lambda J_{s} + (1-\lambda)& J_{s+1} - \frac{1}{2} \lambda (1-\lambda) ||\tilde{x}^s-\tilde{x}^{s+1}||_2^2 - \frac{1}{2}\lambda (1-\lambda) ||\tilde{u}^s-\tilde{u}^{s+1}||_2^2,
\end{align}
for all $\lambda \in [0,1]$. We set $\lambda=1/2$ and obtain 
\begin{equation}
||\tilde{x}^s-\tilde{x}^{s+1}||_2^2 + ||\tilde{u}^s-\tilde{u}^{s+1}||_2^2 \leq 4 (J_s-J_{s+1}).
\end{equation}
According to Prop.~\ref{Prop:AMonDec}, the sequence $J_s$ converges as $s\rightarrow \infty$. As a result, it follows that $(\tilde{x}^s,\tilde{u}^s) \in L^2_n \times L^2_m$ is a Cauchy sequence. The space $L^2_1$ is a Banach space, \cite[p. 67]{rudinRealCAnalysis} and so is $L^2_n \times L^2_m$. Consequently, $(\tilde{x}^s, \tilde{u}^s)$ converges strongly as $s\rightarrow  \infty$, \cite[p. 4]{rudinFA}. The orthonormality of the basis functions implies by Bessel's inequality, \cite[p. 51]{CourantHilberMethods} that the parameters $\eta_{x_s}$ and $\eta_{u_s}$ form a Cauchy sequence in $\ell^2 \times \ell^2$. The square summable sequences form likewise a Banach space and therefore the parameter vectors $\eta_{x_s}$ and $\eta_{u_s}$ converge strongly in $\ell^2$ as $s\rightarrow \infty$. \hfill \qed
\end{proof}

We establish results for the optimization problem \eqref{eq:Prob2}, which are similar to the ones given by Prop.~\ref{Prop:AMonDec}. More precisely, we will show that under Assumptions C0)-C2) the optimal cost of \eqref{eq:Prob2} bounds $J_\infty$ from below and is monotonically increasing in $s$.


\begin{proposition}\label{Prop:BLB}
Let Assumptions C0) and C2) be fulfilled. Then $\tilde{J}_s \leq J_\infty$ holds for all $s \geq 1$.
\end{proposition}
\begin{proof}
We will denote the minimizers of \eqref{eq:Jinf} by $x$ and $u$, which are both square integrable and fulfill $\dot{x}(t)=A x(t)+B u(t)$ for all $t\in [0,\infty)$. We define $\eta_x:=\pi^s(x)\in \mathbb{R}^{ns}$, $\eta_u:=\pi^s(u)\in \mathbb{R}^{ms}$ (where notation is slightly abused to denote both the projection from $L^2_n \rightarrow \mathbb{R}^{ns}$ and the projection from $L^2_m \rightarrow \mathbb{R}^{ms}$, defined in analogy to \eqref{eq:defpi^s}, by $\pi^s$). From Assumption C2) it follows that $\eta_x \in \tilde{\mathcal{X}}^s$ and $\eta_u \in \tilde{\mathcal{U}}^s$. We will argue that $\eta_x$ and $\eta_u$ fulfill the equality constraints in \eqref{eq:Prob2}. Therefore we rewrite the equality constraint as
\begin{equation}
\left( A \otimes I_s - I_n \otimes M\T - I_n \otimes (\tau(0) \tau(0)\T) \right) \eta_x 
+ (B \otimes I_s) \eta_u + (I_n \otimes \tau(0)) x_0 = 0,
\end{equation}
where orthonormality of the basis functions, the properties of the Kronecker product, and Assumption A2) is used. We note further that the identity
\begin{align}
\begin{split}
M=\int_{0}^{\infty} \dot{\tau} \tau\T \dt &= - \tau(0) \tau(0)\T - \int_{0}^{\infty} \tau \dot{\tau}\T \dt \\
&= - \tau(0)\tau(0)\T - M\T,
\end{split}
\end{align}
which follows from integration by parts, simplifies the previous equation to
\begin{align}
(A \otimes I_s + I_n \otimes M)\eta_x
+ (B \otimes I_s) \eta_u + (I_n \otimes \tau(0)) x_0 = 0.\label{eq:cond1}
\end{align}
Moreover, it holds that
\begin{align}
\int_{0}^{\infty} (I_n \otimes \tau) \dot{x}\dt &= -(I_n \otimes \tau(0))x_0 - (I_n \otimes M) \int_{0}^{\infty} (I_n \otimes \tau)x \dt\\
&=-(I_n \otimes \tau(0))x_0 - (I_n \otimes M) \eta_x\label{eq:cond2}\\
&=(A \otimes I_s)\eta_x + (B \otimes I_s) \eta_u, 
\label{eq:cond3}
\end{align}
where integration by parts (1st step), the definition of the projection $\pi^s$ (2nd step), and the fact that $x$ and $u$ fulfill the (linear) equations of motion exactly (3rd step) has been used. Clearly, \eqref{eq:cond2} and \eqref{eq:cond3} are equivalent to \eqref{eq:cond1} and therefore $\eta_x$ and $\eta_u$ are feasible candidates for \eqref{eq:Prob2}. Bessel's inequality, \cite[p. 51]{CourantHilberMethods}, implies that 
\begin{equation}
\frac{1}{2} |\eta_x|_2^2 + \frac{1}{2} |\eta_u|_2^2 = \frac{1}{2} |\pi^s(x)|_2^2 +  \frac{1}{2} |\pi^s(u)|_2^2 \leq \frac{1}{2} ||x||_2^2 + \frac{1}{2} ||u||_2^2,
\end{equation}
where the Euclidean norm is denoted by $|\cdot|_2$.
Therefore $\eta_x$ and $\eta_u$ are feasible candidates achieving a cost that is smaller than $J_\infty$, and hence, $\tilde{J}_s \leq J_\infty$ for all $s \geq s_0$. \hfill \qed
\end{proof}

In order to establish that the sequence $\tilde{J}_s$ is monotonically increasing, we will work with the dual problem. It turns out that the finite dimensional representation of the adjoint equations are fulfilled exactly by \eqref{eq:Prob2}. We will use this fact to construct feasible candidates for the optimization over $s+1$ basis functions.
\begin{proposition}\label{Prop:Bincr}
Let Assumptions C0), C1), and C2) be fulfilled. Then $\tilde{J}_s$ is monotonically increasing and bounded above by $J_\infty$ for all $s\geq 1$.
\end{proposition}

\begin{proof}
We first derive the dual of \eqref{eq:Prob2}. We use Lagrange duality to rewrite \eqref{eq:Prob2} as
\begin{align}
\begin{split}
\tilde{J}_s = &\inf_{\eta_x, \eta_u} \sup_{\eta_p} \int_{0}^{\infty}  \frac{1}{2} \tilde{x}\T \tilde{x} + \frac{1}{2} \tilde{u}\T \tilde{u} + \tilde{p}\T \left( A \tilde{x} + B \tilde{u} - \dot{\tilde x}\right) \dt - \tilde{p}(0)\T (\tilde{x}(0)-x_0),\\
& \quad \eta_x \in \tilde{\mathcal{X}}^s, \eta_u \in \tilde{\mathcal{U}}^s.
\end{split}
\end{align}
From Assumptions C0) and C1) we can infer that $\tilde{J}_s \leq J_{\infty}$ for all $s\geq 1$ by Prop.~\ref{Prop:BLB}. The fact that $0\leq \tilde{J}_s \leq J_{\infty}$ implies further that the infimum in \eqref{eq:Prob2} is attained, and that the set of minimizers is nonempty due to Assumption C0) ($\tilde{\mathcal{X}}^s$ and $\tilde{\mathcal{U}}^s$ are closed). According to \cite[p.~503, Thm.~11.39]{RockafellarConvex} strong duality holds, and the infimum and supremum can be interchanged, which yields
\begin{align}
\begin{split}
\tilde{J}_s = &\sup_{\eta_p} \inf_{\eta_x, \eta_u} \int_{0}^{\infty}  \frac{1}{2} \tilde{x}\T \tilde{x} + \frac{1}{2} \tilde{u}\T \tilde{u} + \tilde{p}\T \left( A \tilde{x} + B \tilde{u} - \dot{\tilde x}\right) \dt - \tilde{p}(0)\T (\tilde{x}(0)-x_0),\\
& \quad \eta_x \in \tilde{\mathcal{X}}^s, \eta_u \in \tilde{\mathcal{U}}^s.
\end{split} \label{eq:minmax}
\end{align}
The minimization over $\eta_x$ and $\eta_u$ is a convex problem and can be rewritten in terms of convex-conjugate functions, \cite[p. 473]{RockafellarConvex}. To that extent, we apply first integration by parts on the term $\tilde{p}\T \dot{\tilde{x}}$, resulting in
\begin{align}
\begin{split}
\tilde{J}_s = &\sup_{\eta_p} \inf_{\eta_x, \eta_u} \int_{0}^{\infty} \frac{1}{2} \tilde{x}\T \tilde{x} + \tilde{x}\T \left( A\T \tilde{p} + \dot{\tilde{p}} \right) + \frac{1}{2} \tilde{u}\T \tilde{u} + \tilde{u}\T B\T \tilde{p}~\dt + \tilde{p}(0)\T x_0,\\
& \quad \eta_x \in \tilde{\mathcal{X}}^s, \eta_u \in \tilde{\mathcal{U}}^s.
\end{split} \label{eq:minmax2}
\end{align}
By defining $\tilde{v}(t,\eta_v):=(I_n \otimes \tau(t))\T \eta_v$ such that 
\begin{equation*}
\int_{0}^{\infty} \delta \tilde{\lambda}\T (\tilde{v} + A\T \tilde{p} + \dot{\tilde{p}} )~\dt=0,\quad \delta \tilde{\lambda}:=(I_n \otimes \tau)\T \delta \eta_\lambda,
\end{equation*}
for all $\delta \eta_\lambda\in \mathbb{R}^{ns}$, which is equivalent to $-\tilde{v}:=A\T \tilde{p} + \dot{\tilde{p}}$ as shown in \cite{parametrizedMPC}, the minimization over $\tilde{x}$ can be interpreted as a (extended real-valued) function of $\tilde{v}$, i.e.
\begin{align}
\inf_{\eta_x \in \tilde{\mathcal{X}}^s}\int_{0}^{\infty} \frac{1}{2} \tilde{x}\T \tilde{x} - \tilde{x}\T \tilde{v}~\dt &=-\sup_{\eta_x \in \tilde{\mathcal{X}}^s}\int_{0}^{\infty} \tilde{x}\T \tilde{v}-\frac{1}{2} \tilde{x}\T \tilde{x}~\dt\\
&=-\sup_{\pi^s(\tilde{x})\in \tilde{\mathcal{X}}^s}\int_{0}^{\infty} \tilde{x}\T \tilde{v}-\frac{1}{2} \tilde{x}\T \tilde{x} ~\dt=:-I_{\varphi_s}^*( \tilde{v} ). \label{eq:defVarphiConj}
\end{align}
Note that $I_{\varphi_s}^*$ maps from $L^2_n$ to the extended real line and is well-defined.
In a similar way, we can regard the minimization over $\eta_u$ as (extended real-valued) function of $\tilde{p}$,
\begin{align}
\inf_{\pi^s(\tilde{u})\in \tilde{\mathcal{U}}^s} \int_{0}^{\infty} \frac{1}{2} \tilde{u}\T \tilde{u}  + \tilde{u}\T B\T \tilde{p}~\dt =: - I_{\psi_s}^*(-B\T\tilde{p}),
\end{align}
where in this case $\pi^s$ denotes the projection $L^2_m \rightarrow \mathbb{R}^{sm}$ defined in analogy to \eqref{eq:defpi^s} (with a slight abuse of notation). Thus, \eqref{eq:minmax2} is reformulated as
\begin{align}
\begin{split}
&\tilde{J}_s = \sup_{\eta_p \in \mathbb{R}^{ns}} -I_{\varphi_s}^*(\tilde{v}) -I_{\psi_s}^*(-B\T \tilde{p}) + \tilde{p}(0)\T x_0,\\
&\text{s.t.} \int_{0}^{\infty} (I_n \otimes \tau) \left( \dot{\tilde{p}}+A\T \tilde{p} + \tilde{v} \right) \dt=0.
\end{split}\label{eq:finiteDual}
\end{align}
The functions $\tilde{v}$ and $\tilde{p}$ satisfy the adjoint equations exactly and it holds that $\lim_{t\rightarrow \infty} \tilde{p}(t)=0$ by Assumption A1).

Let $\eta_v\in \mathbb{R}^{ns}$ and $\eta_p\in \mathbb{R}^{ns}$, with corresponding trajectories $\tilde{v}^s(t,\eta_v)$ and $\tilde{p}^s(t,\eta_p)$, be maximizers of \eqref{eq:finiteDual}. The set of maximizers is non-empty due to the fact that we optimize over $\mathbb{R}^{ns}$ and $0\leq \tilde{J}_s \leq J_\infty$ holds. The equality constraint implies that the adjoint equation $\dot{\tilde{p}}^s(t)+A\T \tilde{p}^s(t) + \tilde{v}^s(t)=0$ is fulfilled for all times $t\in [0,\infty)$, see \cite{parametrizedMPC}, and thus, the adjoint equation is likewise fulfilled by the augmented trajectories $\tilde{v}^{s+1}(t,i_s(\eta_v))$ and $\tilde{p}^{s+1}(t,i_s(\eta_p))$. Hence, $\tilde{v}^{s+1}(t,i_s(\eta_v))$ and $\tilde{p}^{s+1}(t,i_s(\eta_p))$ are feasible candidates to the optimization \eqref{eq:finiteDual} over $s+1$ basis functions, and it holds that $\tilde{p}^{s+1}(0,i_s(\eta_p))=\tilde{p}^s(0,\eta_p)$. It remains to establish the relation between $I_{\varphi_s}^*$ and $I_{\varphi_{s+1}}^*$, as well as $I_{\psi_{s+1}}^*$ and $I_{\psi_s}^*$, which is done via the order reversing property of the convex-conjugation. Therefore the function $I_{\varphi_s}^*$ is regarded as the conjugate of
\begin{equation}
I_{\varphi_s}(x):= \begin{cases} \frac{1}{2} ||\tilde{x}^s(t, \pi^s(x))||_2^2 & \pi^s(x)\in \tilde{\mathcal{X}}^s, \\
\infty  & \text{otherwise}. \end{cases} \label{eq:defvarphi}
\end{equation}
We note that Assumption C1) implies $I_{\varphi_{s+1}}({x}) \geq I_{\varphi_{s}}({x})$ for all ${x} \in L^2_n$. This is due to the fact that any square integrable function $x$ with $\pi^{s+1}({x}) \in \tilde{\mathcal{X}}^{s+1}$ automatically fulfills $\pi^s({x}) \in \tilde{\mathcal{X}}^s$, since $\pi_s(\tilde{\mathcal{X}}^{s+1})$ is contained in $\tilde{\mathcal{X}}^s$ by Assumption C1), and $|\pi^{s+1}(x)|_2^2 \geq |\pi^{s}(x)|_2^2$ holds for all $x \in L^2_n$. The convex-conjugation reverses ordering, which implies
\begin{equation}
I_{\varphi_{s+1}}^*({v}) \leq I_{\varphi_s}^*({v})
\end{equation}
for all ${v}\in L^2_n$, see \cite[Prop.~4.4.1, p.~171]{convexFunctions}. The same reasoning applies to $I_{\psi_s}^*$, which is the convex-conjugate of
\begin{equation}
I_{\psi_s}(u):= \begin{cases} \frac{1}{2} ||\tilde{u}^s(t,\pi^s(u))||_2^2 & \pi^s(u) \in \tilde{\mathcal{U}}^s, \\
\infty & \text{otherwise}. \end{cases} \label{eq:defpsi}
\end{equation}
This leads to the conclusion that 
\begin{equation}
-I_{\varphi_s}^*({v})-I_{\psi_s}^*(-B\T {p}) \leq -I_{\varphi_{s+1}}^*({v}) - I_{\psi_{s+1}}^*(-B\T {p}), \label{eq:dualSol}
\end{equation}
for any ${v}, {p} \in L^2_n$. Hence, we have that $\tilde{v}^{s+1}(t,i_s(\eta_v))$ and $\tilde{p}^{s+1}(t,i_s(\eta_p))$ are feasible candidates to the optimization problem over $s+1$ basis functions with higher corresponding cost and therefore $\tilde{J}_{s+1}\geq \tilde{J}_s$.

\hfill \qed
\end{proof}

Next, we would like to establish that $\lim_{s\rightarrow \infty} \tilde{J}_s =\lim_{s\rightarrow \infty} J_s$. In order to do so, we need the following assumptions:
\begin{itemize}
\item[D0)] $\limsup_{s\rightarrow \infty} \tilde{\mathcal{X}}_s \subset \liminf_{s \rightarrow \infty} \mathcal{X}^s$
\item[D1)] The basis functions $\tau_i$, $i=1,2,\dots$, are dense in $C_0^{\infty}$ (in the topology of uniform convergence).\footnote{The set of smooth functions with compact support mapping from $[0,\infty)$ to $\mathbb{R}$ is denoted by $C_0^{\infty}$.}
\end{itemize}

\begin{proposition}\label{Prop:Conv}
Let $N_0$ be such that $J_{N_0}$ is finite and let Assumptions B0)-D1) be fulfilled. Then, $\lim_{s\rightarrow  \infty} \tilde{J}_{s} = \lim_{s\rightarrow \infty} J_{s}$ holds.
\end{proposition}
\begin{proof}
By Prop.~\ref{Prop:AMonDec} and Prop.~\ref{Prop:AStrongConv} it follows that $J_s$ is monotonically decreasing, $\lim_{s \rightarrow \infty} J_s$ is finite, and that the corresponding optimizers converge. From Prop.~\ref{Prop:BLB} and Prop.~\ref{Prop:Bincr}, we can infer that $\tilde{J}_s$ is monotonically increasing and bounded above by $J_\infty$ for all $s\geq 1$. This implies further that the sequence of minimizers of \eqref{eq:Prob2} is bounded in the $L^2$-sense. Due to the fact that $L^2_1$ (and likewise $L^2_n \times L^2_m$) is a Hilbert space, there exists a subsequence $s(q)$ such that the corresponding minimizer of \eqref{eq:Prob2} converge weakly, i.e. $\tilde{x}^{s(q)} \rightharpoonup \tilde{x}$, $\tilde{u}^{s(q)} \rightharpoonup \tilde{u}$, \cite[p. 163]{convey1990}.

We pick any $\delta p:=(\delta p_1, \dots, \delta p_n)\T$, with $\delta p_i \in C_0^{\infty}$, $i=1,2,\dots,n$, and $\delta p(0)=0$, and choose a sequence $\delta \tilde{p}_k = \sum_{i=1}^{k} \tau_i \delta\eta_{p_i}$, $\delta \eta_{p_i} \in \mathbb{R}^{n}$, converging uniformly to $\delta p$. According to Assumption D1) such a sequence exists. Hence for any $\epsilon > 0$ we can find an integer $N_0$ large enough, such that
\begin{equation*}
|\delta \tilde{p}_k(0)\T (\tilde{x}^{s(q)}(0)-x_0)| \leq |\tilde{x}^{s(q)}(0)-x_0|_2 \epsilon
\end{equation*}
holds for all $k\geq N_0$. We claim that $|\tilde{x}^{s(q)}(0)-x_0|_2$ is uniformly bounded. This can be seen by right multiplying the equality constraint of \eqref{eq:Prob2} by $\eta_x\T$, resulting in
\begin{equation}
\int_{0}^{\infty} {\tilde{x}^{s(q) \TT}} (A \tilde{x}^{s(q)} + B \tilde{u}^{s(q)} - \dot{\tilde{x}}^{s(q)})\dt
- \tilde{x}^{s(q)}(0)\T (\tilde{x}^{s(q)}(0)-x_0)=0,
\end{equation} 
which can be further simplified to
\begin{equation}
\frac{1}{2} |x_0|_2^2+\int_{0}^{\infty} {\tilde{x}^{s(q)\TT}} (A \tilde{x}^{s(q)} + B \tilde{u}^{s(q)} ) \dt
= \frac{1}{2}|\tilde{x}^{s(q)}(0)-x_0|_2^2,
\end{equation}
using $\lim_{t\rightarrow \infty} \tilde{x}^{s(q)}(t)=0$ (by Assumption A1)) and completing the squares.
From the fact that $\tilde{J}_{s(q)} \leq J_\infty$ for all $q$, it follows that $\tilde{x}^{s(q)}$ and $\tilde{u}^{s(q)}$ are bounded in $L^2_n$, respectively $L^2_m$. As a consequence, $|\tilde{x}^{s(q)}(0)-x_0|_2^2$ is uniformly bounded, as can be verified with the Cauchy-Schwarz inequality,
$\lim_{k\rightarrow \infty} \delta \tilde{p}_k(0)\T (\tilde{x}^{s(q)}(0)-x_0)$ converges uniformly, and the limits over $q$ and $k$ can be interchanged,
\begin{align}
\lim_{q\rightarrow \infty} \lim_{k\rightarrow \infty} &\delta \tilde{p}_k(0)\T (\tilde{x}^{s(q)}(0)-x_0) = \lim_{k\rightarrow \infty} \lim_{q\rightarrow \infty} \delta \tilde{p}_k(0)\T (\tilde{x}^{s(q)}(0)-x_0)=0. \label{eq:interchanged}
\end{align}
The equality constraint of \eqref{eq:Prob2} reads therefore as
\begin{align}
\lim_{k \rightarrow \infty} \lim_{q \rightarrow \infty} \int_{0}^{\infty} \delta \tilde{p}_k\T( A \tilde{x}^{s(q)} + B \tilde{u}^{s(q)} - \dot{\tilde{x}}^{s(q)}) \dt = 
\lim_{q \rightarrow \infty} \lim_{k \rightarrow \infty} \int_{0}^{\infty} \delta \tilde{p}_k\T( A \tilde{x}^{s(q)} + B \tilde{u}^{s(q)} - \dot{\tilde{x}}^{s(q)}) \dt =0,
\label{eq:tmp2}
\end{align}
where both limits agree. We will show that the limit over $k$ commutes with the integration. To that extent, we make the following claim: For any function $\tilde{v}:=\sum_{i=1}^N \tau_i \eta_{v_i}$, where $\eta_{v_i}$ are bounded vectors in $\mathbb{R}^n$ and $N$ is a positive integer, it holds that
\begin{equation}
\lim_{k\rightarrow \infty} \int_{0}^{\infty} \delta \tilde{p}_k\T \tilde{v} \dt = \int_{0}^{\infty} \delta p\T \tilde{v}\dt. \label{eq:interclim1}
\end{equation}
We will prove the claim below, but assume for now that it holds. As a consequence of Assumption A1), implying that $\dot{\tilde{x}}^{s(q)}$ is a linear combination of the basis functions, the claim results in
\begin{align}
\lim_{k \rightarrow \infty} \int_{0}^{\infty} \delta \tilde{p}_k \dot{\tilde{x}}^{s(q)}\dt&=\int_{0}^{\infty} \delta p\T \dot{\tilde{x}}^{s(q)} \dt,
\end{align}
for any integer $s(q)$. Using integration by parts (twice) and the fact that $\tilde{x}^{s(q)}$ converges weakly leads to
\begin{align}
\lim_{q\rightarrow \infty} \int_{0}^{\infty} \delta p\T \dot{\tilde{x}}^{s(q)}\dt&=\lim_{q\rightarrow \infty} -\int_{0}^{\infty} \delta \dot{p}\T \tilde{x}^{s(q)}\dt \\
&=-\int_{0}^{\infty} \delta \dot{p}\T \tilde{x} \dt\\
&=\int_{0}^{\infty} \delta p\T \dot{\tilde{x}} \dt.
\end{align}
Note that $\delta \dot{p}$ has compact support, is bounded (by continuity), and is therefore square integrable in $[0,\infty)$.
The claim implies further that \eqref{eq:tmp2} simplifies to
\begin{align}
0&=\lim_{q \rightarrow \infty} \int_{0}^{\infty} \delta p\T( A \tilde{x}^{s(q)} + B \tilde{u}^{s(q)}- \dot{\tilde{x}}^{s(q)})\dt\\
&=\int_{0}^{\infty} \delta p\T (A \tilde{x} + B \tilde{u} - \dot{\tilde{x}})\dt.
\end{align}
The same argument can be repeated for any $\delta p = (\delta p_1, \dots \delta p_n)$, $\delta p_i \in C_0^{\infty}, i=1,2,\dots,n$, vanishing at $0$, and therefore, as $s(q)\rightarrow \infty$, the equality constraint of \eqref{eq:Prob2} reads as
\begin{equation*}
0=\int_{0}^{\infty} \delta p\T (A \tilde{x} + B \tilde{u}- \dot{\tilde{x}})\dt, \quad \forall \delta p \in C_0^{\infty}.
\end{equation*}
Due to the fundamental lemma of the calculus of variations, \cite[p.~18]{YoungOC}, this is equivalent to $\dot{\tilde{x}}(t)=A\tilde{x}(t)+B\tilde{u}(t)$ for all $t\in[0,\infty)$ (almost everywhere). A similar argument based on variations that do not vanish at time $0$ ensures $\lim_{q\rightarrow \infty} \tilde{x}^{s(q)}(0)=x_0$. As a result, the equality constraint of \eqref{eq:Prob2} is equivalent to the one of \eqref{eq:Prob1} in the limit as $s(q)\rightarrow \infty$. Combined with Assumption D1), it implies that $\tilde{x}$ and $\tilde{u}$ are feasible candidates for \eqref{eq:Prob1} and therefore $\lim_{q \rightarrow \infty} \tilde{J}_{s(q)} \geq \lim_{s\rightarrow\infty} J_s$. From Prop.~\ref{Prop:AMonDec}, Prop.~\ref{Prop:BLB}, and Prop.~\ref{Prop:Bincr} it follows that $\tilde{J}_{s}$ is monotonically increasing and bounded by $J_\infty \leq J_s$ for all $s \geq N_0$, resulting in
\begin{equation}
\lim_{q \rightarrow \infty} \tilde{J}_{s(q)} = \lim_{s \rightarrow \infty} \tilde{J}_s = \lim_{s\rightarrow \infty} J_s.
\end{equation}
It remains to prove the claim. Let $\tilde{v}:= \sum_{i=1}^N \tau_i \eta_{v_i}$ where the $\eta_{v_i}$s are bounded vectors in $\mathbb{R}^n$ and $N$ is fixed. The matrix $M$ in Assumption A1) is negative definite and therefore it holds that $|\tilde{v}(t)|_2 \leq C_1 e^{-\beta t}$ for all $t\geq T_0$ for some constants $C_1>0$, $\beta>0$ and time $T_0>0$. As a result, it follows from H\"older's inequality, \cite[p.~76]{rudinRealCAnalysis}, that
\begin{align}
\Big\lvert\int_{0}^{\infty} (\delta \tilde{p}_k - \delta p)\T \tilde{v} \dt  \Big\rvert \leq \sup_{t \in [0,\infty)} |\delta \tilde{p}_k(t) - \delta p(t)|_2 \int_{0}^{\infty} |\tilde{v}|_2~\dt. \label{eq:dpconv}
\end{align}
The second term can be bounded by invoking H\"older's inequality once more,
\begin{align}
\int_{0}^{\infty} |\tilde{v}|_2~\dt &\leq \int_{0}^{T_0} |\tilde{v}|_2~\dt + \int_{T_0}^{\infty} C_1 e^{-\beta t} \dt 
\leq T_0 ||\tilde{v}||_2 + \frac{C_1}{\beta} e^{-\beta T_0}.
\end{align}
Hence, the right-hand side of \eqref{eq:dpconv} converges to zero due to the uniform convergence of the $\delta \tilde{p}_k$ to $\delta p$ as $k \rightarrow \infty$. This proves the claim. \hfill \qed
\end{proof}
\section{Conclusion}
We introduced two different approximations to a class of infinite-horizon optimal control problems encountered in MPC. The approximations bound the optimal cost of the underlying problem from above and below, and their optimal costs converge as the number of basis functions tends to infinity. Under favorable circumstances, the resulting input trajectories of the first approximation are found to approximate the optimal input of the underlying infinite dimensional problem arbitrarily accurately, and the corresponding optimal costs converge to the optimal cost of the underlying infinite dimensional problem. The second approximation yields a lower bound on the cost of the underlying optimal control problem, and can therefore be used to quantify the approximation quality of both approximations.

\section*{Acknowledgment}
The first author would like to thank Jonas L\"uhrmann for a fruitful discussion regarding the proof of Prop.~\ref{Prop:Conv}.

\bibliography{literature}
\bibliographystyle{abbrv}

\end{document}